\documentclass[11pt,a4paper,twoside]{article}
\usepackage{amsfonts} 
\usepackage{amssymb}
\usepackage{amsmath, amsthm, array, ifthen}
\usepackage{mathrsfs} 
\usepackage{caption} 
\usepackage{latexsym}
\usepackage{enumerate}

\usepackage{enumitem}

\usepackage{xfrac}
\usepackage{graphicx}

\usepackage{lmodern}

\usepackage[french]{babel} 
\usepackage[utf8]{inputenc} 

\usepackage{color}

\usepackage[colorlinks]{hyperref}

\newcommand{\RR}{\mathbb{R}}
\newcommand{\CC}{\mathbb{C}}

\newcommand{\vphi}{\varphi}

\newcommand{\mell}[2][f]{\mathcal{M}(f)} 

\newcommand{\diff}{\mathop{}\mathopen{}\mathrm{d}}

\newtheorem{prop}{Proposition}
\newtheorem{prop/not}{Proposition/notation}

\newtheorem{cor}{Corollaire}
\newtheorem{lem}{Lemme}

\theoremstyle{definition}
\newtheorem{rem}{Remarque}

\newcommand{\ioe}{\leqslant}
\newcommand{\soe}{\geqslant}

\usepackage[usenames,dvipsnames,svgnames,table]{xcolor}

\title{Sur la somme de Möbius $\sum_{n \ioe x} \mu(n)n^{-s}$ autour de $s=1$ et des sommes dérivées, première étude}

\usepackage[left=3cm,right=3cm,top=1.5cm,bottom=1.5cm]{geometry}
\author{Daval Florian}
\date{}
\begin{document}
\maketitle
\begin{abstract} Nous prouvons que pour tout $x \soe 10^{12}$ et pour tout $\sigma \soe 1$, on a la majoration \\
$x^{\sigma-1} \log x \big|\sum_{n \ioe x} (\mu(n)/n^{\sigma})\log(x/n)- \log x/\zeta(\sigma)+\zeta'(\sigma)/\zeta^2(\sigma)  \big| \ioe 3.5 \times 10^{-5}$,
améliorant  un résultat de Ramaré et Zuniga-Alterman où à la place de $(3.5 \times 10^{-5};10^{12};\sigma \soe 1)$ ils avaient $(0.043;10^{14}; \sigma \in [1,1.04])$.\\
Nous donnons également d'autres améliorations, en particulier nous étendons des majorations concernant $\sum_{n \ioe x} \mu(n)n^{-s}$ à $\Re(s)>-1$.\\
Nous prouvons  que  $\int_{1}^{x}|\sum_{n \ioe t} \mu(n)/n| \diff{t} \soe 0.0025 (\sqrt{x}-1/x)$ pour tout $x \soe 1$.
\end{abstract}

\section{Introduction}
Les fonctions sommatoires de la fonction de M\"obius
\begin{equation*}
    M(x)=\sum_{n \ioe x} \mu(n)   \quad \text{et } \quad  m(x)=\sum_{n\ioe x} \frac{\mu(n)}{n}
\end{equation*}
et les versions lissées
\begin{equation*}
  \check{m}(x)=\sum_{n\ioe x} \frac{\mu(n)}{n} \log(x/n)=\int_{1}^{x} m(t) \frac{\diff{t}}{t} \quad \text{ et } \quad \check{\check{m}}(x)=\sum_{n\ioe x} \frac{\mu(n)}{n}\log^2(x/n)= 2\int_{1}^{x} \check{m}(t) \frac{\diff{t}}{t} 
\end{equation*}
interviennent dans de nombreux problèmes en théorie des nombres, par exemple dans la résolution complète par Helfgott du problème de Goldbach
ternaire. Il utilise en particulier, dans une de ses prépublications sur le sujet \cite{helfgott2012minor}  une majoration explicite du produit $|m(x)|\log x$ fournie par Ramaré. Pour obtenir des estimations sur $m$ on souhaite utiliser les inégalités obtenues pour $M$ et  des manipulations élémentaires.  

C'est ce que fait El Marraki dans la  prépublication~\cite{preprint}, il a utilisé l'inégalité
\begin{equation*}
|m(x)| \ioe \frac{|M(x)|}{x}+ \frac{1}{x}\int_{1}^{x} |M(t)| \frac{\diff{t}}{t}+\frac{\log x}{x} \quad (x\soe 1)  
\end{equation*}
ce qui est beaucoup plus efficace qu'une intégration par partie.

Balazard dans \cite{BalazarMobiusEnglish}  obtient
\begin{equation}
|m(x)| \ioe \frac{|M(x)|}{x}+ \frac{1}{x^2}\int_{1}^{x} |M(t)| \diff{t}+\frac{8/3}{x} \quad (x\soe 1)\, ,  \label{balineg}
\end{equation}
puis  Ramaré~\cite{RamarExplicitMob} utilise cette dernière inégalité
de Balazard pour convertir ses estimations explicites de $|M|$ en des
estimations explicites pour $|m|$ (la majoration utilisée par Helfgott qui a été évoquée plus haut). 

Dans ma thèse \cite{daval2019identites} j'ai introduit la famille d'identités de convolutions intégrales suivante dont la démonstration est simple (voir la partie \ref{demo} du présent texte ou ma thèse p.16 et p.17 pour une version dans les espaces $\mathrm{L}^p$).
\begin{prop}
Soient $a$ et $b$ deux fonctions arithmétiques,
 $\omega, \varphi: [1, \infty[ \rightarrow \CC$ deux fonctions mesurables. On suppose que $\left.\omega\right|_{[1,x]} $ et
$\left.\varphi\right|_{[1,x]} $ sont bornées. Alors pour tout $x \soe 1$ :
\begin{equation*} 
\!\!\int_{1}^{x}\! \sum_{n \ioe x/t}a(n) \omega \! \Big( \frac{x/t}{n} \Big)  \sum_{k \ioe t}b(k) \varphi\! \Big( \frac{t}{k} \Big) \frac{\diff{t}}{t}
=\!\int_{1}^x \! \sum_{n \ioe x/t}(a\star b)(n) \omega \! \Big( \frac{x/t}{n} \Big)  \varphi (t)\frac{\diff{t}}{t} .
\end{equation*}
\label{terre}
 \end{prop}
Dans ma thèse \cite{daval2019identites} j'ai utilisé cette formule pour les fonctions sommatoires de la fonction de M\"obius $M(x)$ ou $m(x)$ et des versions lissées comme $\check{m}(x)$. J'ai étudié également des identités autour de la fonction de  Liouville $\lambda(n)$ et la fonction de van Mangoldt $\Lambda(n)$. Par exemple, dans la même veine que $\eqref{balineg}$ j'ai obtenu
\begin{equation*}
|\check{m}(x)-1|=\bigg|\sum_{n\ioe x} \frac{\mu(n)}{n} \log(x/n)-1\bigg| \ioe  \frac{1}{x^2}\int_{1}^{x} |m(t)|t \diff{t}+\frac{1.1}{x} \quad (x\soe 1) \,. 
\end{equation*}
Dans \cite{ramar2023mobius} Ramaré et Zuniga-Alterman proposent entre autres de produire des inégalités du type
\begin{equation*}
|\check{m}(x;\sigma)-1|=\bigg|\sum_{n\ioe x} \frac{\mu(n)}{n^{\sigma}} \log(x/n)-\frac{\log x}{\zeta(\sigma)}+\frac{\zeta'(\sigma)}{\zeta(\sigma)}\bigg| \ioe \frac{A_\sigma}{x^{\sigma-1}}  \frac{1}{x}\int_{1}^{x} |m(t)| \diff{t}+\frac{B_\sigma}{x^{\sigma}} \quad (x\soe 1)  \,.
\end{equation*}
 La formule de la proposition \ref{terre} prouvée dans ma thèse comporte 4 paramètres : 2 suites et 2 fonctions. Dans \cite{ramar2023mobius} Ramaré et Zuniga-Alterman utilisent cette formule avec les suites $a(n)=f(n)$, $b(n)=g(n)$ et les fonctions $\vphi(t)=h(1/t)$, $\omega(t)=1$, pour lancer leur fabrique d'identités.  Nous verrons que le paramètre $\omega(t)=\log t$ permettra d'avoir des majorants dépendants non plus de $\int_{1}^{x} |m(t)| \diff{t}$ mais de $\int_{1}^{x} |\check{m}(t)-1| \diff{t}$, ce qui donne de meilleures estimations effectives. 

 Une autre nouveauté dans  \cite{ramar2023mobius} est d'utiliser $a(n)=(-1)^{n+1}$ au lieu de $a(n)=1$ mais nous n'en aurons pas besoin. En particulier dans \cite{ramar2023mobius} la fonction êta de Dirichlet $\eta(s)=(1-2^{1-s}) \zeta(s)$ apparaît à la place de la fonction zêta $\zeta(s)$ (voir notre proposition \ref{mieux} p.~\pageref{mieux} ci-dessous) dans des majorations du type
\begin{equation*}
\bigg|\sum_{n\ioe x} \frac{\mu(n)}{n^{s}}-\frac{1}{\zeta(s)}+\frac{m(x)}{x^{s-1}}\bigg| \ioe \frac{C_s}{x^{\sigma-1}} \frac{|s-1|}{|\eta(s)|}  \frac{1}{x}\int_{1}^{x} |m(t)| \diff{t}+\frac{D_s}{x^{\sigma}} \quad (x\soe 1)  
\end{equation*}
mais $\eta(s)=0$ une infinité de fois pour $\sigma=\Re(s)=1$, ce qui implique que les majorants ci-dessus sont parfois très gros.
  On sait  que Landau s'est beaucoup intéressé aux sommes de Möbius, en particulier à la convergence de $\sum_{n\ioe x} \mu(n)n^{-s}$ vers $1/\zeta(s)$ sur la droite $\Re(s)=1$ et des résultats  équivalents au théorème des nombres premiers par des moyens élémentaires. Il est donc pertinent d'en trouver des versions effectives et élémentaires, ce que nous faisons dans le corollaire \ref{Landau} p.~\pageref{Landau}. Par ailleurs le pôle en $s=1$ de $\zeta(s)$ est intéressant pour nos majorations autour de $s=1$ alors qu'il disparaît avec $\eta(s)=(1-2^{1-s})\zeta(s)$. Signalons que tous les résultats sont pour l'instant bons suivant le paramètre $x$ grand mais faibles par rapport au paramètre $|s|/ \Re(s)$ grand, ce qui doit provenir de la formule d'Euler-Maclaurin et de majorations similaires à $x^{-\Re(s)}=|\int_{x}^{\infty}st^{-(s-1)} \diff{t}| \ioe \int_{x}^{\infty}|st^{s-1}| \diff{t}=x^{-\Re(s)}|s|/\Re(s)$.

 Afin d'avoir des majorations triviales pour pouvoir faire des comparaisons avec les méthodes de convolutions intégrales (c'est-à-dire la proposition \ref{terre}), dans une première étape en utilisant uniquement des intégrations par parties nous prouvons que 
 \begin{align*}
 \bigg|\sum_{n \ioe x} \frac{\mu(n)}{n^\sigma}-\frac{1}{\zeta(\sigma)}+\frac{m(x)}{x^{\sigma-1}}\bigg| &\ioe  \frac{2  (\sigma -1)}{x^{\sigma-1}}  \sup_{t \soe x}|\check{m}(t)-1| \quad (x\soe 1, \sigma>1)\, , \\
  \bigg|\sum_{n \ioe x} \frac{\mu(n)}{n^\sigma}\log(x/n)-\frac{ \log x}{\zeta(\sigma)}+\frac{\zeta'(\sigma)}{\zeta^2(\sigma)}  \bigg| &\ioe  \frac{4 }{x^{\sigma-1}}  \sup_{t \soe x}|\check{m}(t)-1|  \quad (x\soe 1, \sigma>1) \,.
  \end{align*}
Ces majorations triviales surpassent déjà les résultats correspondants dans \cite{ramar2023mobius} car dans \cite{daval2020conversions} j'ai prouvé que 
\begin{equation*}
  |\check{m}(t)-1| \ioe   \frac{8.55 \times 10^{-6}}{\log t} \quad ( t \soe 2.5 \times 10^{11})\,,
\end{equation*}
sans compter que Chirre et Helfgott dans la prépublication \textit{Explicit bounds for sums of the M\"obius function} vont encore améliorer les bornes pour $m$ et $\check{m}$ par des nouvelles méthodes et que leurs résultats seront beaucoup plus forts pour $\check{m}$ que pour $m$. Je remercie H. Helfgoot et A. Chirre pour avoir partagé leur travail avec moi avant sa prépublication.

Dans une deuxième partie je vais cette fois utiliser des identités de convolutions intégrales issue de ma thèse à l'aide de la proposition \ref{terre}  et en particulier étendre les résultats pour $\Re(s)>-1$, ce qui permet pour $s=0$ de retrouver $M(x)$. L'idée est d'utiliser les deux formules suivantes
\begin{equation}
\frac{s-1}{x^{s-1}} \int_{1}^{x}m(x/t) t^{s} \frac{ \diff{t}}{t^2}=\sum_{n \ioe x} \frac{\mu(n)}{n^s}-\frac{m(x)}{x^{s-1}}
\end{equation}
et (par inversion de M\"obius)
\begin{equation}
\frac{s-1}{x^{s-1}}\int_{1}^{x}m(x/t)  \sum_{k \ioe t} \left(\frac{t}{k} \right)^s \frac{ \diff{t}}{t^2}=\frac{s-1}{x^{s-1}}\int_{1}^{x} u^{s} \frac{\diff{u}}{u^2}=1-\frac{1}{x^{s-1}}
\end{equation}
puis d'approcher $ \sum_{k \ioe t} 1/k^{s}$ par $\zeta(s)$ à l'aide de la formule d'Euler-Maclaurin à différents ordres. Pour $s=0$ on retrouve l'étude de la fonction $m_1(x)=m(x)-M(x)/x$ par $m(x)$ comme dans  \cite{daval2020conversions} où j'ai par cette voie prouvé que $\varlimsup |m(x)| \sqrt{x} > \sqrt{2}$ bien que $|m(x)| \sqrt{x}\ioe \sqrt{2}$ pour tout $x \ioe 10^{16}$.

\section{Majorations triviales}  \label{trivial}

\subsection{Les sommes de $\mu(n)n^{-s}$}
Pour étudier $\sum_{n \ioe x} \mu(n)n^{-s}$ et $1/\zeta(s)$ pour $\Re(s)>1$ proche de $1$ on peut utiliser les transformées de Mellin suivantes 
\begin{equation}
 \frac{1}{\zeta(s)}=(s-1) \int_{1}^{\infty}m(t) t^{-s} \diff{t}= (s-1)^2 \int_{1}^{\infty}\check{m}(t) t^{-s} \diff{t}= \frac{ (s-1)^3}{2} \int_{1}^{\infty}\check{\check{m}}(t) t^{-s} \diff{t}
\end{equation}
et leurs versions tronquées
\begin{equation}
(s-1) \int_{x}^{\infty}m(t) t^{-s} \diff{t}= \frac{1}{\zeta(s)}-\sum_{n \ioe x} \frac{\mu(n)}{n^s}+\frac{m(x)}{x^{s-1}}  \label{mtronq}
\end{equation}
et
\begin{equation}
(s-1)^2 \int_{x}^{\infty}(\check{m}(t)-1) t^{-s} \diff{t}= \frac{1}{\zeta(s)}-\sum_{n \ioe x} \frac{\mu(n)}{n^s}+\frac{m(x)}{x^{s-1}}+(s-1)\frac{\check{m}(x)-1}{x^{s-1}}  \label{mtronqch}
\end{equation}
qui font apparaître un développement de type Taylor suivant les puissances de $(s-1)$. Voici une étape supplémentaire :
\begin{align}
&\frac{(s-1)^3}{2} \int_{x}^{\infty}\big( \check{\check{m}}(t)-2\log(t)+2\gamma \big) t^{-s} \diff{t} \nonumber \\
=&\frac{1}{\zeta(s)}-\sum_{n \ioe x} \frac{\mu(n)}{n^s}+\frac{m(x)}{x^{s-1}}+(s-1)\frac{\check{m}(x)-1}{x^{s-1}}+\frac{(s-1)^2}{2}  \frac{\check{\check{m}}(x)-2\log x+2\gamma}{x^{s-1}}\,.  \label{mtronqchch}
\end{align}
On obtient ainsi pour $s=\sigma$ réel la proposition suivante dont le premier résultat à l'avantage d'être uniforme en $\sigma>1$. Les deux autres sont bons au voisinage de $\sigma=1$.
\begin{prop}Pour tout $\sigma>1$ et tout $x \soe 1$ on a
\begin{align*}
\bigg|\frac{1}{\zeta(\sigma)}-\sum_{n \ioe x} \frac{\mu(n)}{n^\sigma}\bigg| &\ioe \frac{2 }{x^{\sigma-1}}  \sup_{t \soe x}|m(t)| \\
\bigg|\frac{1}{\zeta(\sigma)}-\sum_{n \ioe x} \frac{\mu(n)}{n^\sigma}+\frac{m(x)}{x^{\sigma-1}}\bigg| &\ioe  \frac{2  (\sigma -1)}{x^{\sigma-1}}  \sup_{t \soe x}|\check{m}(t)-1| \\
\bigg|\frac{1}{\zeta(\sigma)}-\sum_{n \ioe x} \frac{\mu(n)}{n^\sigma}+\frac{m(x)}{x^{\sigma-1}}+(\sigma-1)\frac{\check{m}(x)-1}{x^{\sigma-1}}\bigg| &\ioe  \frac{ (\sigma -1)^2 }{x^{\sigma-1}} \sup_{t \soe x}|\check{\check{m}}(t)-2\log t+2\gamma| \,.
\end{align*} \label{prop1}
\end{prop}
Remarquons que si l'on exprime l'inverse de la fonction zêta avec $M$ plutôt qu'avec $m$ on a $1/\zeta(s)=s \int_{1}^{\infty}M(t) t^{-s-1} \diff{t}$, ce qui est moins intéressant au voisinage de $s=1$ :
\begin{equation*}
\Big| \frac{1}{\zeta(\sigma)}-\sum_{n \ioe x} \frac{\mu(n)}{n^{\sigma}}+\frac{M(x)}{x^{\sigma}}\Big| \ioe  \frac{\sigma}{\sigma-1} \sup_{t \soe x} \frac{|M(t)|}{t}  \frac{1}{x^{\sigma-1}} \quad (x\soe 1, \sigma>1) \,.
\end{equation*}

\subsection{Les sommes de $\mu(n)n^{-s}\log(x/n)$}
Pour étudier $\sum_{n \ioe x}\mu(n)n^{-s}\log(x/n)$ il suffit de dériver les formules de la sous-section précédente. En dérivant la formule \eqref{mtronq} multipliée par $x^s$  par rapport à $s$, puis en divisant par $x^s$, on trouve
\begin{equation} \label{deriv K1}
 \int_{x}^{\infty}m(t) t^{-s} \diff{t}+(s-1) \int_{x}^{\infty}m(t) t^{-s} \log(x/t) \diff{t}= \frac{ \log x}{\zeta(s)}-\frac{ \zeta'(s)}{\zeta^2(s)}-\sum_{n \ioe x} \frac{\mu(n)}{n^s}\log(x/n) \,,
\end{equation}
en utilisant plutôt $\check{m}$ et la formule \eqref{mtronqch} on a
\begin{align}  \label{deriv K2}
&2(s-1) \int_{x}^{\infty}(\check{m}(t)-1) t^{-s}\diff{t}+(s-1)^2 \int_{x}^{\infty}(\check{m}(t)-1)  t^{-s} \log(x/t) \diff{t}  \nonumber \\ 
&=\frac{ \log x}{\zeta(s)}-\frac{ \zeta'(s)}{\zeta^2(s)}-\sum_{n \ioe x} \frac{\mu(n)}{n^s}\log(x/n) +\frac{\check{m}(x)-1}{x^{s-1}} \, ,
\end{align}
puis avec $\check{\check{m}}$ et  la formule \eqref{mtronqchch} on obtient
\begin{align} \label{deriv K3}
& \frac{3(s-1)^2}{2}\int_{x}^{\infty}\big( \check{\check{m}}(t)-2 \log(t)+2\gamma \big) t^{-s} \diff{t}  \nonumber \\ 
&\hspace{3cm}+ \frac{(s-1)^3}{2}\int_{x}^{\infty}\big( \check{\check{m}}(t)-2 \log(t)+2\gamma \big) t^{-s} \log(x/t) \diff{t}  \nonumber \\ 
=&\frac{ \log x}{\zeta(s)}-\frac{ \zeta'(s)}{\zeta^2(s)}-\sum_{n \ioe x} \frac{\mu(n)}{n^s}\log(x/n) +\frac{\check{m}(x)-1}{x^{s-1}}+(s-1)  \frac{\check{\check{m}}(x)-2\log x+2\gamma}{x^{s-1}} \,.
\end{align}
On obtient ainsi pour $s$ réel la proposition suivante qui est moins bonne au voisinage de $\sigma=1$ que la proposition \ref{prop1} d'un facteur $\sigma -1$. 
\begin{prop}Pour tout $\sigma>1$ et tout $x \soe 1$ on a
\begin{align*}
\bigg|\frac{ \log x}{\zeta(\sigma)}-\frac{ \zeta'(\sigma)}{\zeta^2(\sigma)}-\sum_{n \ioe x} \frac{\mu(n)}{n^\sigma}\log(x/n) \bigg| &\ioe \frac{\frac{2}{\sigma-1}}{x^{\sigma-1}}  \sup_{t \soe x}|m(t)| \\
\bigg|\frac{ \log x}{\zeta(\sigma)}-\frac{\zeta'(\sigma)}{\zeta^2(\sigma)} -\sum_{n \ioe x} \frac{\mu(n)}{n^\sigma}\log(x/n) \bigg| &\ioe  \frac{4 }{x^{\sigma-1}}  \sup_{t \soe x}|\check{m}(t)-1| \\
\bigg|\frac{ \log x}{\zeta(\sigma)}-\frac{\zeta'(\sigma)}{\zeta^2(\sigma)}+\frac{\check{m}(x)-1}{x^{\sigma-1}} -\sum_{n \ioe x} \frac{\mu(n)}{n^\sigma}\log(x/n) \bigg| &\ioe  \frac{3 (\sigma -1) }{x^{\sigma-1}}  \sup_{t \soe x}|(\check{\check{m}}(x)-2\log x+2\gamma)| \,.
\end{align*}
\end{prop}
Si l'on ne se place pas au voisinage de $\sigma=1$ il y a des avantages à utiliser la deuxième majoration qui est uniforme en $\sigma$. C'est ce que nous avons fait pour obtenir le résultat annoncé dans le résumé qui améliore celui de \cite{ramar2023mobius}.
\begin{rem}
    Dans \cite{ramar2023mobius} ce sont les sommes $\sum_n\mu(n)n^{-s}\log(x/n)$ avec $n$ vérifiant $n \wedge q=1$ et $n\ioe x$ qui sont étudiées. Mais d'une part toutes les majorations triviales vues précédemment s'étendent sans modification à ces conditions de co-primalité. D'autre part c'est le cas $q=1$ qui est le plus important car on peut en suivant la méthode du livre d'Helfgott (disponible sur sa page internet), passer du cas $q=1$ au cas général avec peu de pertes. Je renvoie à \cite{ramar2023mobius} pour plus de détails.
\end{rem}
\section{Majorations obtenues par des convolutions intégrales}
\subsection{En passant par $m$ ou $\check{m}$}
Les majorations triviales de la section précédente ne s'appliquent que pour $\Re(s)>1$, pour passer à gauche de cette ligne je vais utiliser les méthodes de convolutions introduites dans ma thèse. Les résultats sont intéressants pour des $s$ avec des  parties imaginaires relativement petites.
\begin{prop} \label{mieux}
Soit $s\neq 1$ un nombre complexe  on a :
\begin{equation}
\zeta(s) \bigg( \sum_{n \ioe x} \frac{\mu(n)}{n^s} -\frac{m(x)}{x^{s-1}} \bigg)-1=\frac{\check{m}(x)-1}{x^{s-1}} +\frac{1}{x^{s-1}}\int_{1}^{x}m(x/t) Q_s(t)\frac{\diff{t}}{t^2}
\end{equation}
avec  \begin{equation}
Q_s(t)=(s-1)\zeta(s)t^{s}- (s-1) \sum_{k \ioe t} \left(\frac{t}{k} \right)^s-t
\end{equation}
et pour $\Re(s)>0$ on a
\begin{align}
\zeta(s)&\bigg(\sum_{n \ioe x} \frac{\mu(n)}{n^s}-\frac{m(x)}{x^{s-1}}-(s-1)\frac{\check{m}(x)-1}{x^{s-1}} \bigg)-1=\frac{(s-1)}{x^{s-1}}\big(\frac{1}2\check{\check{m}}(x)-\log x+\gamma\big)\\
&+\frac{ (s-1)}{x^{s-1}}\int_{1}^{x}  [\check{m}(x/t)-1 ] Q_s(t) \frac{\diff{t}}{t^2}
-(s-1)^2\int_{x}^{\infty}  \Big[ \log t -
\sum_{j \ioe t} \frac{1}{j} -\gamma \Big] t^{-s}\diff{t} \,. \nonumber
\end{align}
Si  $\sigma=\Re(s)>0$ alors $|Q_s(t)| \ioe \dfrac{ |s|}{|\sigma |} |s-1|$ et si de plus $\zeta(s) \neq 0$ alors 
\begin{equation}  \label{parm}
 \bigg|   \sum_{n \ioe x} \frac{\mu(n)}{n^s} -\frac{m(x)}{x^{s-1}}-\frac{1}{\zeta(s)} \bigg| \ioe  \dfrac{|s|}{|\sigma |} \dfrac{|s-1|}{|\zeta(s)|} \frac{1}{x^{\sigma -1}} \frac{1}{x} \int_{1}^{x}|m(t)| \diff{t}+\frac{|\check{m}(x)-1|}{|\zeta(s)| x^{\sigma-1}} 
\end{equation}
et
\begin{align}  \label{parchm}
\bigg|   \sum_{n \ioe x} \frac{\mu(n)}{n^s} -\frac{m(x)}{x^{s-1}}-(s-1)\frac{\check{m}(x)-1}{x^{s-1}}-\frac{1}{\zeta(s)} \bigg| \ioe  \dfrac{|s|}{|\sigma |} \dfrac{|s-1|^2}{|\zeta(s)|} \frac{1}{x^{\sigma -1}} \frac{1}{x} \int_{1}^{x}|\check{m}(t)-1| \diff{t}&\\
 +\frac{|s-1|}{|\zeta(s)| x^{\sigma-1} }\big|\frac{1}2\check{\check{m}}(x)-\log x+\gamma\big|+\frac{0.55|s-1|^2}{|\zeta(s)||\sigma| x^{\sigma}}&\,. \nonumber
\end{align}
\end{prop}
Dans \cite{ramar2023mobius} les auteurs obtiennent le même type de majorations que \eqref{parm} mais avec  $\eta(s)=(1-2^{1-s}) \zeta(s)$  au lieu de $\zeta(s)$, ce qui d'une part supprime le pôle de $\zeta$ en $s=1$ alors qu'il est avantageux pour l'inégalité au voisinage de $s=1$. D'autre part $1-2^{1-s}=0$ pour la suite $s_n=1+2n\pi i/\log(2)$ ce qui est dommage pour l'étude sur la droite $\Re(s)=1$ que nous allons voir dans le corollaire suivant. (On a $s_1 \approx 1+9i$.)
\begin{cor} \label{Landau}
    Supposons que $s=\rho$ soit un zéro de la fonction zêta alors pour tout $x\soe 1$ on a  \begin{equation}
    \int_{1}^{x}|m(t)| \diff{t} \soe \left(1+ \frac{ |\rho-1| |\rho|}{|\Re(\rho) |}\right)^{-1} (x^{\Re(\rho)}-1/x)
\end{equation}
En particulier 
\begin{equation}
    \int_{1}^{x}|m(t)| \diff{t} \soe 0.0025 (\sqrt{x}-1/x)\,.
\end{equation}
De plus si l'on suppose  $\Re(\rho)=1$ et $|m(x)| =\varepsilon+o(1)$ alors 
\begin{equation}|\Im(\rho)| \soe \frac{1}{\sqrt{\varepsilon}}-1\,. \end{equation}
L'asymptotique $m(x)=o(1)$ entraîne que $\zeta(s)\neq 0$ pour $\Re(s)=1$ et on a
\begin{equation}
 \bigg|   \sum_{n \ioe x} \frac{\mu(n)}{n^{1+it}}-\frac{1}{\zeta(1+it)}  \bigg| \ioe   \dfrac{2+t^2}{|\zeta(1+it)|}  \frac{1}{x} \int_{1}^{x}|m(u)| \diff{u}+|m(x)|+\frac{1}{x^2}\quad (x\soe 1) \,.
\end{equation}
\end{cor}
\begin{rem}[\textbf{et remerciements}] 
Dans ma thèse \cite{daval2019identites} j'ai montré un résultat similaire pour $M$ (théorème 4 p.119)
\begin{equation*}
\int_{1}^{x} |M(t)| \frac{\diff{t}}{t}  \soe 0{.}004 \sqrt{x}  \quad (x \soe 2)\,.
\end{equation*} 
Olivier Bordellès et indépendemment Olivier Ramaré m'ont donné une piste en 2020 pour améliorer la constante du résultat de ma thèse ci-dessus. Je vais développer ceci dans le prochain papier car cela nécessite quelques calculs numériques (voir aussi la partie expérimentale \ref{normes}). Je les remercie tous les deux pour cette indication et pour avoir pris en considération les travaux de ma thèse puis pour m'avoir encouragé. Plus précisément dans \cite{dona2022explicit} on trouve 
\begin{equation} \label{Hel}
  \bigg| \zeta(s)-  \sum_{n \ioe t} \frac{1}{n^s}-\frac{t^{s-1}}{s-1} \Bigg| \ioe \frac{5/6}{t^{\sigma}} \quad (t \soe |\Im(s)|, 0<\Re(s) \ioe 1, s \neq 1)
\end{equation}
qui appliqué pour le zéro de la fonction zêta de plus petite hauteur $s=1/2+14.13...i$ donnera un meilleure constante environ 18 fois plus petite que $0.0025$ et $0.004$ pour les minorations en $\sqrt{x}$. Il faudra compléter le résultat avec une étude numérique avec contrôle de l'erreur pour $t \ioe |\Im(s)|$ et adapter certaines démonstrations.
\end{rem}
\begin{proof}[Démonstration du corollaire \ref{Landau}]
Appliquons l'égalité de la proposition \ref{mieux} avec $\zeta(\rho)=0$
\begin{equation*}
-x^{\rho-1}=(\check{m}(x)-1) +\int_{1}^{x}m(x/t) Q_{\rho}(t)\frac{\diff{t}}{t^2}\,.
\end{equation*} 
Par le lemme \ref{balcheck} on a $x |\check{m}(x)-1| \ioe\int_{1}^{x}|m(t)| \diff{t}+1/x$ donc
\begin{align*}
x^{\Re(\rho)-1}&= \bigg|(\check{m}(x)-1) +\int_{1}^{x}m(x/t) Q_{\rho}(t)\frac{\diff{t}}{t^2} \bigg| \\
&\ioe  \frac{1}{x} \int_{1}^{x}|m(t)| \diff{t} + \frac{1}{x^2}+   \dfrac{ |\rho|}{|\Re(\rho) |} |\rho-1| \int_{1}^{x}|m(x/t)|\frac{\diff{t}}{t^2} \,. 
\end{align*}
Ainsi
\begin{equation} \label{mino}
    \int_{1}^{x}|m(t)| \diff{t} \soe \left(1+ \frac{ |\rho|  |\rho-1|}{|\Re(\rho) |}\right)^{-1} (x^{\Re(\rho)}-1/x) 
\end{equation}
et pour le zéro de plus petite hauteur ($\rho \approx 1/2+14.13i$) on obtient
\begin{equation*}
\int_{1}^{x}|m(t)| \diff{t} \soe 0.0025 (\sqrt{x}-1/x) \,.
\end{equation*}

Si $\Re(\rho)=1$ et $|m(x)| \ioe \varepsilon+o(1)$ alors par \eqref{mino} pour tout $x \soe 1$ on a
$$1+|\rho-1|^2\soe 1+ |\rho|  |\rho-1| \soe \frac{1}{\varepsilon+o(1)}(1-1/x^2)$$
et on obtient le résultat en faisant tendre $x$ vers $+\infty$.

Pour finir si $m(x)=o(1)$, puisque $1/\sqrt{\varepsilon} \rightarrow +\infty$ quand $\varepsilon \rightarrow 0$ alors $\zeta(s) \neq 0$ pour $\Re(s)=1$. On peut donc appliquer la deuxième partie de la proposition \ref{mieux}.
\end{proof}
Quand on développe
\begin{equation}
Q_s(t)=(s-1)\zeta(s)t^{s}- (s-1) \sum_{k \ioe t} \left(\frac{t}{k} \right)^s-t
\end{equation}
suivant la formule d'Euler-Maclaurin le terme principal est $(s-1)(1/2-\{t\})$ et on a
$$\int_{1}^{x}m(x/t)(s-1)(1/2-\{t\}) \frac{\diff{t}}{t^2}=(s-1)\big(\check{m}(x)-1)-m_1(x)/2+1/x\big)$$
on obtient donc la proposition suivante où la majoration est en $x^{-2}\int_{1}^{x}|m(t)|t \diff{t}$ au lieu de $x^{-1}\int_{1}^{x}|m(t)| \diff{t}$.
\begin{prop} \label{poids}
Soit $s\neq 1$ un nombre complexe  on a :
\begin{equation*}
 \zeta(s) \bigg( \sum_{n \ioe x} \frac{\mu(n)}{n^s} -\frac{m(x)}{x^{s-1}} \bigg)-1=\frac{s(\check{m}(x)-1)}{x^{s-1}}-\frac{(s-1)m_1(x)}{2x^{s-1}}+\frac{(s-1)}{{x^{s}}} +\frac{1}{x^{s-1}}\int_{1}^{x}m(x/t) R_s(t)\frac{\diff{t}}{t^2}
\end{equation*}
avec  \begin{equation}
R_s(t)=(s-1)\zeta(s)t^{s}- (s-1) \sum_{k \ioe t} \left(\frac{t}{k} \right)^s-t-(s-1)(\{t\}-1/2)\,.
\end{equation}

Si $\Re(s)>-1$ alors on a la majoration $|R_s(t)| \ioe  \dfrac{|s+1|}{|\sigma+1|} \dfrac{ |s-1||s|}{6 }  \dfrac{1}{t}$.

Si $s=\sigma$ est réel et $\sigma >-1$ alors
\begin{align*}
 \bigg|   \sum_{n \ioe x} \frac{\mu(n)}{n^\sigma}-\frac{1}{\zeta(\sigma)} -\frac{m(x)}{x^{\sigma-1}}\bigg|
 \ioe  &\dfrac{|\sigma||\sigma-1|}{8|\zeta(\sigma)|} \frac{1}{x^{\sigma -1}} \frac{1}{x^2} \int_{1}^{x}|m(t)|t \diff{t} \\
 &+\frac{1}{|\zeta(\sigma)| x^{\sigma-1}}\left(\frac{ |\sigma-1|}{2}|m_1(x)|+ |\sigma||\check{m}(x)-1|+\frac{ |\sigma-1|}{x} \right) 
\end{align*}
où $m_1(x)=x^{-1}\int_{1}^{x}m(t) \diff{t}$.
\end{prop}
Les termes $m_1(x)$ et $\check{m}(x)-1$ ne posent pas de problèmes car en partant du résultat de Ramaré \cite{ramare2013explicit} 
\begin{equation} \label{RamareM}
|M(x)|  \ioe 0.013 x/\log x\quad(x \soe 97\,067)
\end{equation}
j'ai prouvé dans    \cite{daval2020conversions}  que
\begin{align*}
 |m(x)|\ioe&   0.0130073/ \log x \quad (x\soe 97\,063) \\
 |\check{m}(x)-1| \ioe&   (8.55 \times 10^{-6})/\log x \quad ( x \soe 2.5 \times 10^{11})\, , \\
 |m_1(x)| \ioe&  (7.265\times 10^{-6})/\log x \quad (x \soe 2.15 \times 10^{11}) \,.
\end{align*} 
Et si l'on préfère les inégalités intégrales (pour d'autres résultats que le théorème des nombres premiers) on a 
\begin{align*}
    |\check{m}(x)-1| &\ioe \frac{1}{x}\int_{1}^{x}|m(t)| \diff{t}+ \frac{1}{x^2}, \quad |\check{m}(x)-1| \ioe \frac{1}{x^2}\int_{1}^{x}|m(t)|t \diff{t}+ \frac{1.1}{x}, \\
\text{ et }  \quad \quad   |m_1(x)| &\ioe \frac{1}{x}\int_{1}^{x}|m(t)| \diff{t}+ \frac{1}{x}, \quad \quad \quad |m_1(x)| \ioe \frac{1}{x^2}\int_{1}^{x}|m(t)|t \diff{t}+ \frac{2}{x}\,.
\end{align*}

\subsection{Par la série harmonique et autres idées au voisinage de $s=1$}

\subsubsection{La série harmonique}
De la même manière que nous avons utilisé $m(u)=\sum_{n \ioe u}\mu(n)/n$ plutôt que $M(u)=\sum_{n \ioe u} \mu(n)$ pour exprimer $1/\zeta(s)$, on peut utiliser  $H(u)=\sum_{n \ioe u} 1/n$ pour exprimer la fonction $\zeta(s)$. Dans \cite[lemma 2.8]{dona2022explicit} on trouve une preuve rigoureuse de 
\begin{equation}
-0.5407...=-2(\log 2+\gamma -1)   \ioe  x(H(x)-\log x-\gamma) \ioe \frac{1}{2} \quad (x \soe 1)\, .
\end{equation}
Par ailleurs on a 
\begin{equation}
(s-1) \int_{t}^{\infty}(H(u)-\log u -\gamma) u^{-s} \diff{u}=\zeta(s)-\sum_{n \ioe t}\frac{1}{n^s}-\frac{1}{(s-1)t^{s-1}}+\frac{H(t)-\log t -\gamma}{t^{s-1}} \label{har}
\end{equation}
\begin{equation}
s\int_{t}^{\infty}(\lfloor u\rfloor-u-1/2)u^{-s-1} \diff{u}=\zeta(s)-\sum_{n \ioe t} \frac{1}{n^s}-\frac{1}{(s-1)t^{s-1}}+\frac{\lfloor t\rfloor-t+1/2}{t^s}\,. \label{ent}
\end{equation}
C'est le facteur $s$ devant l'intégrale dans \eqref{ent} qui se retrouve dans la majoration $$|Q_s(t)-(s-1)(\{t\}-1/2)| \ioe  0.5  |s-1| |s| /|\sigma |$$ et dans les propositions \ref{mieux} et \ref{poids}, si on utilise \eqref{har} on obtient $$|Q_s(t)-(s-1)t(H(t)-\log t-\gamma)| \ioe 0.55  |s-1|^2/|\sigma |\,. $$ Pour se localiser encore plus en $s=1$ on peut envisager de passer par 
\begin{equation*} \check{H}(x)= \int_{1}^{x} H(t) \frac{\diff{t}}{t}=\sum_{n \ioe x} \frac{1}{n} \log(x/n) \quad \text{ et } \quad \check{\check{H}}(x)= 2\int_{1}^{x} \check{H}(t) \frac{\diff{t}}{t} \,. \end{equation*}
Pour adapter la proposition \ref{poids} on aurait besoin de 
\begin{equation} \label{double check borne}
\int_{1}^{x}m(x/t)t(H(t)-\log t-\gamma) \frac{\diff{t}}{t^2}=-0.5(\check{\check{m}}(x)-2\log x +2\gamma)-\gamma(\check{m}(x)-1) \,.
\end{equation}

\subsubsection{Normes intégrales à la place des normes infinies et positivité} \label{normes}
Dans la proposition \ref{mieux}, au lieu d'utiliser des majorations de la norme infinie de $Q_s$ comme  $|Q_s(t)| \ioe    |s-1| |s| /|\sigma |$ on peut passer par des normes $\mathrm{L}^1$, par exemple $\int_{1}^{\infty}|Q_s(t)|t^{-2}\diff{t}$, en effet j'ai montré dans \cite{daval2020conversions} que c'était beaucoup plus efficace pour des résultats équivalents au théorème des nombres premiers (voir les résultats sur $\check{m}$ et $m_1$ juste après \eqref{RamareM}, les inégalités avec intégrales proviennent de normes infinies et les autres de normes $\mathrm{L}^1$). Cela peut se voir simplement en remplaçant $m$ par $\varepsilon$ dans 
\begin{equation*}
\zeta(s) \bigg( \sum_{n \ioe x} \frac{\mu(n)}{n^s} -\frac{m(x)}{x^{s-1}} \bigg)-1=\frac{\check{m}(x)-1}{x^{s-1}} +\frac{1}{x^{s-1}}\int_{1}^{x}m(x/t) Q_s(t)\frac{\diff{t}}{t^2}\,.
\end{equation*}
On a 
\begin{equation}
    \int_{1}^{\infty}Q_s(t)\frac{\diff{t}}{t^2}=\frac{1}{s-1}-\zeta(s)+\gamma
\end{equation}
et si l'on suppose qu'il n'y a pas trop de compensations c'est le bon ordre de grandeur de $\int_{1}^{\infty}|Q_s(t)|t^{-2}\diff{t}$. Mais l'ordre de grandeur semble bien petit. Ce calcul exact d'intégrale permet tout de même de contrôler les calculs numériques en PARI/GP et les problèmes de précisions.

Attention, dans la fin de cette sous-section il y a des tests mais pas de démonstrations rigoureuses.

Pour $s=0.5+14.13i$ par la mojoration de la proposition \ref{mieux} on a $|Q_s(t)| \ioe 399.9$.  Mais  par \eqref{Hel} on a $|Q_s(t)| \ioe 9.4$ pour $t \soe 14.13$ et des calculs PARI/GP semblent montrer que  $|Q_s(t)| \ioe 20.512$ pour $t \ioe 14.13$, ce qui donnerait
\begin{equation}
    \int_{1}^{x}|m(t)| \diff{t} \soe 0.047 (\sqrt{x}-1/x)\,.
\end{equation}
Il semble aussi que
\begin{equation}
    \int_{1}^{\infty}|Q_s(t)|\frac{\diff{t}}{t^2}\ioe  \int_{1}^{20}|Q_s(t)|\frac{\diff{t}}{t^2}+\frac{9.4}{20}\ioe 11.
\end{equation} 
Il faudrait donc réussir à mieux étudier $Q_s$ pour des $s$ complexes avec disons $|s|/ \Re(s) \soe 10$.  Rappelons que la fonction \textit{intnum} de  PARI/GP ne donne pas de borne pour l'erreur commise par l'approximation.
\section{Fin des démonstrations} \label{demo}

\begin{proof}[Démonstration de la proposition \ref{terre}]
Utilisons les notations du chapitre IV du livre~\cite{balazardlivre} de Balazard, définissons $
\mathrm{S}_a \phi :[1,\infty[ \rightarrow \CC$ par  \[\mathrm{S}_a \vphi(x)=\sum_{n \ioe x}a(n) \vphi(x/n) \quad(x \soe 1).
\]
Le  résultat suivant correspond à l'égalité (3) de~\cite{balazardlivre} page 35, on a
\begin{equation}
\mathrm{S}_{a } \circ \mathrm{S}_{b } = \mathrm{S}_{a \star b}. \label{conv}
\end{equation}  
Nous utiliserons  les deux interversions entre somme et intégrale suivantes : 
\[
\int_{1}^x \sum_{n \ioe x/t} \cdots \diff{t}\underset{(\nabla)}{=}\sum_{n\ioe x} \int_{1}^{x/n}\cdots \diff{t} 
\quad \text{ et } \quad \int_{1}^x \sum_{n \ioe t}\cdots\diff{t}\underset{(\Delta)}{=}\sum_{n\ioe x} \int_{n}^{x} \cdots \diff{t}. 
\] 
On a la suite d'égalité suivante : 
\begin{align*} 
\int_{1}^{x} \mathrm{S}_a \omega (x/t) \mathrm{S}_b \vphi (t)  \frac{\diff{t}}{t}\underset{(\nabla)}{=}\sum_{n \ioe x} a(n)\int_{1}^{x/n} \omega (x/(tn)
\mathrm{S}_b \vphi (t)    \frac{\diff{t}}{t}&\underset{\phantom{(\nabla)}}{=}\sum_{n \ioe x} a(n) \int_{n}^{x} \omega (x/u)  \mathrm{S}_b \vphi (u/n)  
\frac{\diff{u}}{u} \\
&\underset{(\Delta)}{=}\int_{1}^{x}  \omega (x/u)    \sum_{n \ioe u} a(n)\mathrm{S}_b \vphi (u/n)   \frac{\diff{u}}{u},
\end{align*} 
où l'on a utilisé le changement de variable $t=u/n$.\\
L'application de $\eqref{conv}$ donne
$\sum_{n \ioe u} a(n)\mathrm{S}_b \vphi (u/n)=\mathrm{S}_{a \star b}\vphi (u)$, ce qui conduit à
\begin{align*}
\int_{1}^{x}  \omega (x/u)    \sum_{n \ioe u} a(n)\mathrm{S}_b \vphi (u/n)   \frac{\diff{u}}{u}&\underset{\phantom{(\nabla)}}{=}\int_{1}^{x} \omega (x/u) 
\mathrm{S}_{a \star b}\vphi (u) \frac{\diff{u}}{u}\\
&\underset{(\Delta)}{=}\sum_{n \ioe x} (a \star b)(n) \int_{n}^{x} \omega(x/u)\vphi (u/n)  \frac{\diff{u}}{u}\\
&\underset{\phantom{(\nabla)}}{=}\sum_{n \ioe x} (a \star b)(n) \int_{1}^{x/n}\omega(x/(nt))\vphi (t)  \frac{\diff{t}}{t} \\
&\underset{(\nabla)}{=} \int_{1}^{x} \sum_{n \ioe x/t} (a \star b)(n) \omega(x/(nt)) \vphi (t)  \frac{\diff{t}}{t},
\end{align*}
où l'on  a encore une fois utilisé le changement de variable $t=u/n$.
On obtient l'égalité annoncée.
\end{proof}

\begin{lem} \label{Abel}
Soit $a$ une suite arithmétique, on pose $L(t) =\sum_{n \ioe t } a(n)/n$. Pour tout nombre complexe $s$ et tout $x \soe 1$ on a
\begin{equation}
   (s-1) \int_{1}^{x} \sum_{n \ioe t } L(t) t^{-s} \diff{t}=\sum_{n \ioe x} \frac{a(n)}{n^s}- \frac{L(x)}{x^{s-1}} \,.
\end{equation}
\end{lem}
\begin{proof}
    On écrit
    $$\sum_{n \ioe x} \frac{a(n)}{n^s}= \sum_{n \ioe x} \frac{a(n)}{n} n^{-s+1}$$
    et on utilise la formule sommatoire d'Abel.
\end{proof}

\begin{lem} \label{int check}
    Soit  $f:[1,+\infty[ \mapsto \CC$ une fonction intégrable et posons  $\check{f}(x)= \int_{1}^{x} f(t)/t \diff{t}$ alors pour tout  $s\in \CC$ on a 
    \begin{equation*}
        (s-1) \int_{1}^{x}  \check{f}(t)  t^{-s} \diff{t}=  \int_{1}^{x} f(t)t^{-s} \diff{t} - x^{1-s}  \check{f}(x) \quad (x \soe 1) \,.
    \end{equation*}
\end{lem}
\begin{proof}
Si $s\neq1$ alors par le théorème de Fubini on a
    \begin{align*}
   \int_{1}^{x}  \int_{1}^{t} f(u) \frac{\diff{u}}{u} t^{-s} \diff{t}= \int_{1}^{x}  \int_{u}^{x} t^{-s} \diff{t} f(u)\frac{\diff{u}}{u}= \frac{1}{s-1} \int_{1}^{x} f(u)u^{-s} \diff{u}- \frac{x^{1-s}}{s-1} \int_{1}^{x} f(u) \frac{\diff{u}}{u} \,.
\end{align*}
Si $s=1$ alors les deux membres de l'égalité valent $0$.
\end{proof}

\begin{proof}[Démonstrations de la partie triviale \ref{trivial}]
Commençons par les formules passant par $m$. On applique le lemme \ref{Abel} avec $a(n)=\mu(n)$ et on  obtient
\begin{equation}
(s-1) \int_{1}^{x}m(t) t^{-s} \diff{t}=\sum_{n \ioe x} \frac{\mu(n)}{n^s}-\frac{m(x)}{x^{s-1}}  \label{Mellin tronq m}
\end{equation}
si de plus $\Re(s)>1$ on peut passer à la limite car $|m(t)| \ioe 1$ pour tout $t \soe 1$ et
\begin{equation}
(s-1) \int_{1}^{\infty}m(t) t^{-s} \diff{t}=\frac{1}{\zeta(s)} \,.
\end{equation}
  En soustrayant les deux égalités précédentes on obtient
\begin{equation}
(s-1) \int_{x}^{\infty}m(t) t^{-s} \diff{t}= \frac{1}{\zeta(s)}-\sum_{n \ioe x} \frac{\mu(n)}{n^s}+\frac{m(x)}{x^{s-1}} \,.
\end{equation}
La majoration annoncée en découle par la formule suivante appliquée pour $s=\sigma$ 
\begin{equation} \label{int moins s}
\int_{x}^{\infty} t^{-s} \diff{t}=\frac{1}{s-1} \frac{1}{x^{s-1}} \,.
\end{equation}

Passons aux  formules passant par $\check{m}$. On applique le lemme \ref{int check} avec $f(t)=m(t)$ puis l'égalité \eqref{Mellin tronq m}, on obtient
\begin{align} \label{Mellin tronq check m}
(s-1)^2 \int_{1}^{x}\check{m}(t) t^{-s} \diff{t}&=(s-1)\int_{1}^{x}m(t) t^{-s}- (s-1)x^{1-s}\check{m}(x) \nonumber \\
&=\sum_{n \ioe x} \frac{\mu(n)}{n^s}-\frac{m(x)}{x^{s-1}}-(s-1)\frac{\check{m}(x)}{x^{s-1}} 
\end{align}
si de plus $\Re(s)>1$ on peut passer à la limite car $|\check{m}(t)-1| \ioe 2$ pour tout $t \soe 1$ (par le lemme \ref{balcheck})
\begin{equation}
(s-1)^2 \int_{1}^{\infty}\check{m}(t) t^{-s} \diff{t}=\frac{1}{\zeta(s)} \,.
\end{equation}
Par soustraction on aboutit à 
\begin{equation}
(s-1)^2 \int_{x}^{\infty}\check{m}(t) t^{-s} \diff{t}= \frac{1}{\zeta(s)}-\sum_{n \ioe x} \frac{\mu(n)}{n^s}+\frac{m(x)}{x^{s-1}}+ (s-1)\frac{\check{m}(x)}{x^{s-1}} \,.
\end{equation}
Par \eqref{int moins s} on a 
\begin{equation}
(s-1)^2 \int_{x}^{\infty}(\check{m}(t)-1) t^{-s}-1 \diff{t}=(s-1)^2 \int_{x}^{\infty}\check{m}(t) t^{-s} \diff{t}-(s-1)  
\end{equation}
et donc 
\begin{equation}
(s-1)^2 \int_{x}^{\infty}(\check{m}(t)-1) t^{-s} \diff{t}= \frac{1}{\zeta(s)}-\sum_{n \ioe x} \frac{\mu(n)}{n^s}+\frac{m(x)}{x^{s-1}}+(s-1)\frac{\check{m}(x)-1}{x^{s-1}} \,.
\end{equation}

Passons aux  formules passant par $\check{\check{m}}$. On applique le lemme \ref{int check} avec $f(t)=\check{m}(t)$ en sachant que $\check{\check{m}}(x)=2 \check{f}(t)$, puis l'égalité \eqref{Mellin tronq check m}, on obtient
\begin{align} 
(s-1)^3 \int_{1}^{x}\frac{1}{2}\check{\check{m}}(t) t^{-s} \diff{t}&=(s-1)^2\int_{1}^{x}\check{m}(t) t^{-s}- (s-1)^2x^{1-s}\frac{1}{2}\check{\check{m}}(x) \nonumber \\
&=\sum_{n \ioe x} \frac{\mu(n)}{n^s}-\frac{m(x)}{x^{s-1}}-(s-1)\frac{\check{m}(x)}{x^{s-1}}-\frac{(s-1)^2}{2} \frac{\check{\check{m}}(x)}{x^{s-1}} \,.
\end{align}
Si de plus $\Re(s)>1$ on peut passer à la limite car $|\check{\check{m}}(t)-\log(t)+\gamma| \ioe 4\gamma+2$ pour tout $t \soe 1$ (voir \eqref{double check borne})
\begin{equation}
\frac{(s-1)^3}{2} \int_{1}^{\infty}\check{\check{m}}(t) t^{-s} \diff{t}=\frac{1}{\zeta(s)} \,.
\end{equation}
Par soustraction on aboutit à 
\begin{equation}
\frac{(s-1)^3}{2}  \int_{x}^{\infty}\check{\check{m}}(t) t^{-s} \diff{t}= \frac{1}{\zeta(s)}-\sum_{n \ioe x} \frac{\mu(n)}{n^s}+\frac{m(x)}{x^{s-1}}+ (s-1)\frac{\check{m}(x)}{x^{s-1}} +\frac{(s-1)^2}{2} \frac{\check{\check{m}}(x)}{x^{s-1}} \,.
\end{equation}
Pour finir 
\begin{align*}
&(s-1)^3 \int_{x}^{\infty}\frac{1}{2}\big( \check{\check{m}}(t)+a \log(t)+b \big) t^{-s} \diff{t}\\
=&\frac{1}{\zeta(s)}-\sum_{n \ioe x} \frac{\mu(n)}{n^s}+\frac{m(x)}{x^{s-1}}+(s-1)\frac{\check{m}(x)+a/2}{x^{s-1}}+\frac{(s-1)^2}{2}  \frac{\check{\check{m}}(x)+a\log x+b}{x^{s-1}}
\end{align*}
et on choisit $a=-2$ puis $b=2\gamma$.

Nous allons maintenant dériver les transformées de Mellin tronquée qui sont toutes de la forme 
\begin{equation}
(s-1)^K \int_{x}^{\infty}k(t) t^{-s} \diff{t}= F(s)=\frac{1}{\zeta(s)}-\sum_{n \ioe x} \frac{\mu(n)}{n^s}+ \sum_{k=0}^{K-1}\frac{(s-1)^k}{k!}\frac{f_k(x)}{ x^{s-1}} \,.
\end{equation}
On a
\begin{equation}
K(s-1)^{k-1}\int_{x}^{\infty}k(t)(x/t)^s \diff{t}+(s-1)^K\int_{x}^{\infty}k(t)\log(x/t)(x/t)^s=\big(x^sF(s)\big)'
\end{equation}
et
\begin{equation}
x^{-s}\big(x^sF(s)\big)'=\frac{\log x}{\zeta(s)}-\frac{\zeta'(s)}{\zeta^2(s)}-\sum_{n \ioe x} \frac{\mu(n)}{n^s} \log(x/n)+\sum_{k=1}^{K-1}\frac{(s-1)^{k-1}}{(k-1)!}\frac{f_k(x)}{ x^{s-1}} \,.
\end{equation}
En faisant $K=1$ on obtient \eqref{deriv K1}, $K=2$ donne \eqref{deriv K2} et $K=3$ fournit \eqref{deriv K3}.

On  multiplie \eqref{int moins s} par $x^s$ puis on dérive par rapport à $s$ et on divise par $x^s$ : 
\begin{equation} \label{int moins s log}
\int_{x}^{\infty} t^{-s} \log(x/t) \diff{t}=\frac{-1}{(s-1)^2} \frac{1}{x^{s-1}} \,.
\end{equation}
La majoration annoncée en découle par la formule ci-dessus appliquée  pour $s=\sigma$.
\end{proof}

\begin{lem} \label{balcheck}
Pour tout $x\soe 1$ on a
    $$ |\check{m}(x)-1| \ioe \frac{1}{x}\int_{1}^{x}|m(t)| \diff{t}+ \frac{1}{x^2}\,.$$
\end{lem}
\begin{proof}

Dans la proposition \ref{terre} on prend $a=\mu$, $\omega(t)=t$, $\vphi(t)=2/t$ et $b(n)=1$, on a
\begin{equation}
\int_{1}^{x} m(x/t) \bigg( \frac{1}{t^2}\sum_{k \ioe t} 2k \bigg) \frac{\diff{t}}{t}
= 2\int_{1}^x   \frac{\diff{t}}{t^3}=1-\frac{1}{x^2}\;. \label{formule m}
\end{equation} 
Balazard dans l'appendice \cite{ramar2023mobius} et dans \cite{BalazardHal} introduit la fonction $\alpha$ par 
    $$\frac{1}{t^2}\sum_{k \ioe t} 2k=1+\alpha(t)$$
    et il prouve en particulier que $|\alpha(t)| \ioe 1/t$. Or $\int_{1}^{x} m(x/t)/t  \diff{t}= \check{m}(x)$,  ainsi 
    \begin{equation*}
         \check{m}(x)-1= \int_{1}^{x} m(x/t) \alpha(t)  \frac{\diff{t}}{t}-\frac{1}{x^2}
    \end{equation*}
    ce qui donne le résultat annoncée par inégalité triangulaire.
\end{proof}

\begin{proof}[Démonstration des propositions \ref{mieux} et \ref{poids}, cas $m(x/t)$]
On part du lemme \ref{Abel} appliqué avec $a(n)=\mu(n)$ et du changement de variable $u=x/t$,
\begin{equation}
\sum_{n \ioe x} \frac{\mu(n)}{n^s}-\frac{m(x)}{x^{s-1}}=(s-1)\int_1^x m(u)u^{-s} \diff{u}=\frac{s-1}{x^{s-1}} \int_{1}^{x}m(x/t) t^{s-2} \diff{t} \,.
\end{equation}
Puis par la proposition \ref{terre} appliquée avec $a=\mu$, $\omega(t)=t$, et $b(n)=1$ on obtient
\begin{equation}
\int_{1}^{x} m(x/t)\frac{1}{t}\sum_{k \ioe t} \vphi\! \left( \frac{t}{k} \right) \frac{\diff{t}}{t}
= \int_{1}^x   \vphi (t)\frac{\diff{t}}{t^2}
\end{equation}
et donc 
\begin{equation}
(s-1)\int_{1}^{x}m(x/t)  \sum_{k \ioe t} \left(\frac{t}{k} \right)^s \frac{ \diff{t}}{t^2}=(s-1)\int_{1}^{x} u^{s} \frac{\diff{u}}{u^2}=x^{s-1}-1 \,.
\end{equation}
On a aussi 
\begin{equation}
    (s-1)\int_{1}^{x}m(x/t) t \frac{ \diff{t}}{t^2}=\sum_{n \ioe x} \frac{\mu(n)}{n} \log(x/n)=\check{m}(x) \,.
\end{equation}
Rappelons que 
 \begin{equation}
Q_s(t)=(s-1)\zeta(s)t^{s}- (s-1) \sum_{k \ioe t} \left(\frac{t}{k} \right)^s-t
\end{equation}
et les formules précédentes permettent d'obtenir une expression pour
$$\frac{1}{x^{s-1}} \int_{1}^{x}m(x/t) Q_s(t) t^{s-2} \diff{t}$$
qui est la première identité de la proposition \ref{mieux}.

Passons aux majorations. On trouvera dans \cite{Balazard2008} de Balazard  p.~34 la formule suivante
\begin{equation*}
(s-1) \sum_{k \ioe t} \left(\frac{t}{k} \right)^s =-t+(s-1)\zeta(s)t^s+(s-1)(1/2-\{t\})+t^{s}(s-1)s\int_{t}^{\infty}(\{u\}-1/2)u^{-s-1} \diff{u}\;, \label{F zeta}
\end{equation*}
elle provient de la formule d'Euler-Maclaurin. On dispose des majorations
\begin{equation}
  \bigg|  t^{s} (s-1)s\int_{t}^{\infty}(\{u\}-1/2)u^{-s-1} \diff{u} \bigg| \ioe  \frac{|s|}{\sigma} \frac{|\sigma-1|}{2}   \quad (s \in \CC, \Re(s)>0)
\end{equation}
ou par la seconde formule de la moyenne et (voir livre \cite{balazardlivre} p.11 et p.17 pour l'utilisation de $ -1/8\ioe \int_{1}^{X}(\{u\}-1/2) \diff{u}=(\{X\}^2-\{X\})/2 \ioe 0$)
\begin{equation}
  \bigg|  t^{\sigma} (\sigma-1)\sigma  \int_{t}^{\infty}(\{u\}-1/2)u^{-\sigma-1} \diff{u} \bigg|  \ioe |\sigma| |\sigma-1| \frac{1}{8t}   \quad (\sigma \in \RR, \sigma>-1) \, .
\end{equation}
On peut aussi continuer avec une intégration par partie 
\begin{equation}
 t^{s} (s-1)s \int_{t}^{\infty}\frac{\{u\}-1/2}{u^{s+1}} \diff{u}= \frac{(s-1)s B_2(\{ t\})}{2 t}-\frac{(s-1)s(s+1) t^{s}}{2 } \int_t^{\infty}   \frac{ B_2(\{u\}) }{u^{s+2}} \diff{u} 
\end{equation}
et ainsi
\begin{equation}
 \bigg| t^{s} (s-1)s \int_{t}^{\infty}(\{u\}-1/2)u^{-s-1} \diff{u} \bigg|  \ioe \frac{|s+1|}{|\sigma+1|} \frac{ |\sigma-1||\sigma|}{6 }  \frac{1}{t} \quad (s \in \CC, \Re(s)>-1)\,.
\end{equation}

\end{proof}

Pour entrevoir les généralisations dans les deux lemmes suivants nous obtenons les formules pour tout entier $k$ bien que dans la fin de ce papier seul $k=1$ soit utilisé. Posons $$\check{m}_k(x)=\sum_{n\ioe x} \frac{\mu(n)}{n}\log^k(x/n) \,.$$
Nous avons vu que $\check{m}_1(x)=\check{m}(x)=O(1)$ et $\check{m}_2(x)=\check{\check{m}}(x)=2\log x+O(1)$. Le lemme suivant, prouvé par Shapiro dans l'article~\cite{shapiro1950theorem} sans avoir recours au théorème des nombres premiers, généralise ces trois cas.
\begin{lem}[Shapiro] Soit $k\soe 1$ un entier, on a $\check{m}_k(x)=Q_k(\log x)+O(1)$ avec $Q_k(X)$ un polynôme de degré $k-1$ qui a pour terme dominant
$kX^{k-1}$.
\end{lem}
\begin{proof} Shapiro \cite{shapiro1950theorem}  montre que $\sum_{n \ioe x} (\mu(n)/n) \log^k(x/n)=\sum_{i=1}^{k-1}c_i^{(k)}\log^i x+O(1)$ avec $c_{k-1}^{(k)}=k$.
\label{Shapiro}
\end{proof}

\begin{lem} \label{k gen}
Soit $k$ un nombre entier. Pour tout $x\soe1$ on a
 \begin{align*}
\frac{1}{x^{s-1}}\int_{1}^{x}  \Big[\check{m}_k(x/t)-Q_k(\log (x/t)) \Big]  \sum_{j \ioe t}  \left( \frac{t}{j} \right)^s \frac{\diff{t}}{t^2} 
=\int_{1}^x  \Big[ \log^k \left( t \right)-
\sum_{j \ioe t} \frac{1}{j}Q_k \Big( \log  \left(\frac{t}{j} \right)\!\!\Big)  \Big] t^{-s}\diff{t} \,.
\end{align*}
\end{lem}

\begin{proof}
La proposition \ref{terre} utilisée avec $a=\mu$, $\omega(t)=t \log^k t  $ et $b=1$ fournit l'identité
\begin{equation}
\int_{1}^{x}  \check{m}_k(x/t)  \sum_{k \ioe t} \vphi\! \left( \frac{t}{k} \right) \frac{\diff{t}}{t^2}
=\int_{1}^x  \log^k \left( x/t \right)  \vphi (t)\frac{\diff{t}}{t^2} \,.
\end{equation}
En prenant $a=1\star \mu$ et $b=1$ dans la proposition \ref{terre} on obtient
\begin{equation}
\int_{1}^{x}\omega \left( \frac{x}{t}  \right) \sum_{k \ioe t}\vphi\left(\frac{t}{k}\right) \frac{\diff{t}}{t}=\int_{1}^{x}\vphi\left( \frac{x}{t}  \right) \sum_{k \ioe t}\omega \left(\frac{t}{k}\right) \frac{\diff{t}}{t}\;, \label{voyage}
\end{equation}
avec $\omega(t)=tQ_k (\log t) $ et $\vphi(t)=t^s$ on arrive à 
\begin{equation*}
\int_{1}^{x}  \frac{x}{t}  Q_k(\log(x/t)) \sum_{j \ioe t}  \left( \frac{t}{j} \right)^s  \frac{\diff{t}}{t}=
\int_{1}^{x} (x/t)^s    \sum_{j \ioe t} \frac{t}{j} Q_k\big( \log\left( \frac{t}{j} \right) \big) \frac{\diff{t}}{t}
\end{equation*} 
et obtient l'identité annoncée par soustraction.
\end{proof}

\begin{proof}[Démonstration des propositions \ref{mieux} et \ref{poids}, cas $\check{m}(x/t)$]
On utilise le lemme \ref{k gen} avec $k=1$, on a
\begin{align} \label{k 1}
\frac{1}{x^{s-1}}\int_{1}^{x}  [\check{m}(x/t)-1 ]  \sum_{j \ioe t}  \left( \frac{t}{j} \right)^s \frac{\diff{t}}{t^2} 
=\int_{1}^x  \Big[ \log t -
\sum_{j \ioe t} \frac{1}{j}  \Big] t^{-s}\diff{t}\,.
\end{align}
On a également
\begin{equation} \label{check m integ double}
   \int_{1}^{x}  [\check{m}(x/t)-1 ]  t \frac{\diff{t}}{t^2}= \frac{1}2\check{\check{m}}(x)-\log x \,.
\end{equation}
Par  l'égalité \eqref{Mellin tronq check m} on a
\begin{equation} \label{tr}
\frac{(s-1)^2}{x^{s-1}} \int_{1}^{x}\check{m}(x/t) t^{s-2} \diff{t}=\sum_{n \ioe x} \frac{\mu(n)}{n^s}-\frac{m(x)}{x^{s-1}}-(s-1)\frac{\check{m}(x)}{x^{s-1}}\,.
\end{equation}
Posons
\begin{equation}
q_s(t)=(s-1)^2\zeta(s)t^{s}- (s-1)^2 \sum_{k \ioe t} \left(\frac{t}{k} \right)^s-t(s-1)=(s-1)Q_s(t)
\end{equation}
on utilise successivement  les trois égalités \eqref{tr}, \eqref{k 1} et \eqref{check m integ double}   on a
\begin{align*}
&\frac{1}{x^{s-1}}\int_{1}^{x}  [\check{m}(x/t)-1 ]  q_s(t) \frac{\diff{t}}{t^2}=\\
&\zeta(s)\bigg(\sum_{n \ioe x} \frac{\mu(n)}{n^s}-\frac{m(x)}{x^{s-1}}-(s-1)\frac{\check{m}(x)}{x^{s-1}} \bigg)+(s-1)(-1+x^{1-s})\zeta(s)\\
&-(s-1)^2\int_{1}^x  \Big[ \log t -
\sum_{j \ioe t} \frac{1}{j}  \Big] t^{-s}\diff{t}\\
&-\frac{(s-1)}{x^{s-1}}(\frac{1}2\check{\check{m}}(x)-\log x)
\end{align*}
or
\begin{equation}
    -(s-1)^2\int_{1}^x  \Big[ \log t -
\sum_{j \ioe t} \frac{1}{j}  \Big] t^{-s}\diff{t}=(s-1)^2\int_{x}^{\infty}  \Big[ \log t -
\sum_{j \ioe t} \frac{1}{j}  \Big] t^{-s}\diff{t}+(s-1)\zeta(s)-1
\end{equation}
donc
\begin{align*}
\zeta(s)&\bigg(\sum_{n \ioe x} \frac{\mu(n)}{n^s}-\frac{m(x)}{x^{s-1}}-(s-1)\frac{\check{m}(x)-1}{x^{s-1}} \bigg)-1=\frac{(s-1)}{x^{s-1}}\big(\frac{1}2\check{\check{m}}(x)-\log x-\gamma\big)\\
&+\frac{ (s-1)}{x^{s-1}}\int_{1}^{x}  [\check{m}(x/t)-1 ] Q_s(t) \frac{\diff{t}}{t^2}
-(s-1)^2\int_{x}^{\infty}  \Big[ \log t -
\sum_{j \ioe t} \frac{1}{j} -\gamma \Big] t^{-s}\diff{t} \,.
\end{align*}

\end{proof}

\bibliographystyle{alpha}

\bibliography{bibarticle}

\end{document}